\documentclass{amsart}

 \usepackage{amssymb,latexsym,amsmath,amsthm}

 \renewcommand{\div}{\mathop{\mathrm{div}}\nolimits}

\usepackage{xypic}
\usepackage[all]{xy}

\usepackage{url}

\def\eps{\varepsilon}

\def\R{{\mathbb R} }

\def\r{\R^{n+1}_{+}}
\def\rn{\R^{n}}
\def\br{\partial\r}
\def\super{\overline}

\numberwithin{equation}{section}

\usepackage[numbers,sort&compress]{natbib}

\newtheorem{thm}{Theorem}[section]

\newtheorem{dfn}{Definition}[section]

\newtheorem{lemma}{Lemma}[section]

\newtheorem{remark}{Remark}[section]

\newtheorem{cor}{Corollary}[section]

\newtheorem{notation}{Notation}[section]

\begin{document}

\title{On stable solutions of the fractional H\'{e}non-Lane-Emden equation}

\author{Mostafa Fazly} 

\address{Department of Mathematical and Statistical Sciences, CAB 632, University of Alberta, Edmonton, Alberta, Canada T6G 2G1}

\email{fazly@ualberta.ca}

 \author{Juncheng Wei}
 
\address{Department of Mathematics, University of British Columbia, Vancouver, B.C. Canada V6T 1Z2.} 

\email{jcwei@math.ubc.ca}

\thanks{Both authors are partially supported by NSERC  grants.}

\maketitle

\begin{abstract} We derive monotonicity formulae for solutions of the fractional H\'{e}non-Lane-Emden equation \begin{equation*}
(-\Delta)^{s} u=|x|^a |u|^{p-1} u \ \ \ \text{in } \ \ \mathbb{R}^n, 
\end{equation*}
when $0<s<2$, $a>0$ and $p>1$. Then,  we apply these formulae to classify stable solutions of the above equation.
\end{abstract}

\tableofcontents

\section{Introduction and main results}
We study the classification of stable solutions of the following equation
\begin{equation}\label{main}
(-\Delta)^{s} u= |x|^a |u|^{p-1} u \ \ \ \text{in } \ \ \mathbb{R}^n, 
\end{equation}
where  $(-\Delta)^{s}$  is the fractional Laplacian operator for $0<s<2$.  Let us first provide the definition of stability.  
\begin{dfn}
We say that a solution $u$ of (\ref{main}) is stable  if
\begin{equation}\label{stability} \int_{\mathbb R^n} \int_{\mathbb R^n} \frac{ ( \phi(x)-\phi(y) )^2 }{|x-y|^{n+2s}} dx dy-p \int_{\mathbb R^n} |x|^a |u|^{p-1} \phi^2 \ge 0,
\end{equation}
 for any $\phi\in C_c^\infty(\mathbb R^n)$.
 \end{dfn}
For the local operator cases $s=1$, Laplacian operator, and $s=2$, bi-Laplacian operator, the classification of stable solutions  is completely known for $a\ge 0$. For the case of Laplacian operator that is when $s=1$, we refer interested readers to Farina \cite{f} when $a=0$ and to Cowan-Fazly \cite{cf}, Wang-Ye \cite{wy}, Dancer-Du-Guo \cite{dyg}, Du-Guo-Wang \cite{dgw} when  $a>-2$. In addition, for the fourth order Lane-Emden equation that is when $s=2$ we refer to Davila-Dupaigne-Wang-Wei \cite{ddww} when  $a=0$ and to Hu \cite{hu} when  $a>0$.   In this paper, we consider the fractional Laplacian operator $(-\Delta)^s$ when $0<s<2$ and $s\neq1$.  Note that $0<s<1$ and $1<s<2$ the classification of finite Morse index solutions of (\ref{main}) for $a=0$ are given by Davila-Dupaigne-Wei in \cite{ddw}  and by Fazly-Wei in \cite{fw}, respectively.

It is by now standard that the fractional Laplacian can be seen as a Dirichlet-to-Neumann operator for a degenerate but {\it local} diffusion operator in the higher-dimensional half-space $\R^{n+1}_{+}$. For the case of $0<s<1$ this in fact can be seen as the following theorem given by Caffarelli-Silvestre \cite{cs}, see also \cite{s}.
\begin{thm}\label{extension}
Take $s\in(0,1)$, $\sigma>s$ and $u\in C^{2\sigma}(\R^n)\cap L^1(\R^n,(1+\vert t\vert)^{n+2s}dt)$.
For $X=(x,y)\in\R^{n+1}_{+}$, let
\begin{equation}
 u_e(X)= \int_{\R^n} P(X,t)u(t)\;dt,
\end{equation}
where
\begin{equation}
P(X,t) = p_{n,s}\, t^{2s}\vert X-t\vert^{-(n+2s)}, 
\end{equation}
and $p_{n,s}$ is chosen so that $\int_{\R^n}P(X,t)\;dt=1$.
Then, $ u_e\in C^2(\r)~\cap~C(\super\r)$, $y^{1-2s}\partial_{y} u_e\in C(\super\r)$ and
\begin{equation}
\left\{
\begin{aligned}
\nabla\cdot(y^{1-2s}\nabla  u_e)&=0&\quad\text{in $\R^{n+1}_{+}$,}\\
 u_e&= u&\quad\text{on $\br$,}\\
-\lim_{y\to0} y^{1-2s}\partial_{y} u_e&= \kappa_s(-\Delta)^s u&\quad\text{on $\br$,}
\end{aligned}
\right.
\end{equation}
where
\begin{equation} \label{kappas}
\kappa_s = \frac{\Gamma(1-s)}{2^{2s-1} \Gamma(s)}.
\end{equation}
\end{thm}
Applying the above theorem to the fractional Henon-Lane-Emden equation (\ref{main}), we get the following equation in $\R^{n+1}_{+}$,
\begin{equation} \label{maine}
\left\{
\begin{aligned}
-\nabla \cdot ( y^{1-2s} \nabla  u_e ) &= 0 &\quad \text{in } \quad \r,  \\
-\lim_{y\to0} y^{1-2s} \partial_t  u_e &= \kappa_s |x|^a \vert  u_e\vert^{p-1} u_e
&\quad \text{in } \quad \mathbb R^n. 
\end{aligned}
\right.
\end{equation}
There are different ways of defining the fractional operator $(-\Delta)^{s}$ where $1<s<2$, just like the case of $0<s<1$. Applying the Fourier transform one can define the fractional Laplacian by
\begin{equation}
\widehat{ (-\Delta)^{s}}u(\zeta)=|\zeta|^{2s} \hat u(\zeta)  , 
\end{equation}
 or equivalently define this operator inductively by  $(-\Delta)^{s}=  (-\Delta)^{s-1} o (-\Delta)$, see \cite{rs}.   Recently, Yang in \cite{y} gave a characterization of the fractional Laplacian $(-\Delta)^{s}$, where $s$ is  any positive, noninteger number as the Dirichlet-to-Neumann map for a function $u_e$ satisfying a  higher order elliptic equation in the upper half space with one extra spatial dimension. This is a generalization of the work of Caffarelli and Silvestre in \cite{cs} for the case of $0<s<1$. We first fix the following notation then we present the Yang's characterization. See also Case-Chang \cite{cc} and Chang-Gonzales \cite{cg} for higher order fractional operators.

\begin{notation} Throughout this note set $b:=3-2s$ and define the operator 
\begin{equation} 
\Delta_b w:=\Delta w+\frac{b}{y} w_y=y^{-b} \div(y^b \nabla w),  
\end{equation}
for a function $w\in W^{2,2}(\mathbb R^{n+1},y^b)$.
\end{notation}

 As it is shown by Yang in \cite{y}, if $u(x)$ is a solution of (\ref{main}) then the extended function $u_e(x,y)$ where $x\in\mathbb R^n$ and $y\in\mathbb R^+$ satisfies
\begin{eqnarray}\label{maine}
 \left\{ \begin{array}{lcl}
\hfill \Delta^2_b u_e&=& 0   \ \ \text{in}\ \ \mathbb{R}^{n+1}_{+},\\
\hfill  \lim_{y\to 0} y^{b}\partial_y{u_e}&=& 0   \ \ \text{in}\ \ \partial\mathbb{R}^{n+1}_{+},\\
\hfill \lim_{y\to 0} y^{b} \partial_y \Delta_b u_e &=& C_{n,s} |x|^a |u|^{p-1} u  \ \ \text{in}\ \ \mathbb{R}^{n} . 
\end{array}\right.
\end{eqnarray}
Moreover, 
\begin{equation}
\int_{\mathbb R^n} |\xi|^{2s} |\hat{u(\xi)}|^2 d\xi=C_{n,s} \int_{\mathbb {R}^{n+1}_+} y^b |\Delta _b u_e(x,y)|^2 dx dy. 
\end{equation}     
Note that  $u(x)=u_e(x,0)$ in $\mathbb R^n$.  On the other hand,  Herbst in \cite{h} (see also \cite{ya}), shoed that when $n>2s$ the following Hardy inequality holds
\begin{equation} 
\int_{\mathbb R^n} |\xi|^{2s} |\hat \phi|^2 d\xi > \Lambda_{n,s} \int_{\mathbb R^n} |x|^{-2s} \phi^2 dx , 
\end{equation}
for any $\phi \in C_c^\infty(\mathbb R^n)$ where the optimal constant given by 
\begin{equation} 
 \Lambda_{n,s}=2^{2s}\frac{ \Gamma(\frac{n+2s}{4})^2  }{ \Gamma(\frac{n-2s}{4})^2}.\end{equation}
Here we fix a constant that plays an important role in the  classification of solutions of (\ref{main})
\begin{equation}
p_{S}(n,a)=\left\{
\begin{aligned}
+\infty&\quad\text{if $n\le 2s$}   , \\
\frac{n+2s+2a}{n-2s}&\quad\text{if $n> 2s$}. 
\end{aligned}
\right.
\end{equation}

\begin{remark}
Note that  for $p>p_{S}(n,a)$ the function
\begin{equation}
\label{sing sol}
u_s(x) = A |x|^{-\frac{2s+a}{p-1}} , 
\end{equation}
where
\begin{equation}
A^{p-1} = \lambda\left(\frac{n-2s}{2} -\frac{2s+a}{p-1}\right) , 
\end{equation}
for constant \begin{equation}
\lambda(\alpha) = 2^{2s} \frac{\Gamma(\frac{n+2s+2\alpha}{4}) \Gamma(\frac{n+2s-2\alpha}{4})}{\Gamma(\frac{n-2s-2\alpha}{4})\Gamma(\frac{n-2s+2\alpha}{4})} , 
\end{equation}
is a singular solution of (\ref{main}) where $0<s<2$.  For more details, we refer interested readers to \cite{fall} for the case of $0<s<1$ and to \cite{fw} for the case of $1<s<2$.
\end{remark}
We now present our main result; 
\begin{thm}\label{mainthm}
\label{thmstable}
Assume that $n\ge1$ and $0<s<\sigma<2$.
Let $u\in C^{2\sigma}(\R^n)\cap L^1(\R^n,(1+\vert y\vert)^{n+2s}dy)$ be a stable solution to (\ref{main}). Then for Sobolev subcritical exponents $1<p<p_{S}(n,a)$ and for Sobolev supercritical exponents $p_{S}(n,a)<p$ such that 
\begin{equation} \label{cond}
p \frac{\Gamma(\frac n2-\frac{s+\frac{a}{2}}{p-1}) \Gamma(s+\frac{s+\frac{a}{2}}{p-1})}{\Gamma(\frac{s+\frac{a}{2}}{p-1}) \Gamma(\frac{n-2s}{2} - \frac{s+\frac{a}{2}}{p-1})} > \frac{\Gamma(\frac{n+2s}{4})^2}{\Gamma(\frac{n-2s}{4})^2},
\end{equation}
 the solution $ u$ must be identically zero.   For the case of Sobolev critical exponent that is when  $p=p_{S}(n,a)$,  the solution $u$ has finite energy i.e.
\begin{equation} 
   \| u\|_{\dot H^s(\R^n)}^2 =\int_{\mathbb R^{n}}|x|^a |u|^{p+1}<+\infty.
   \end{equation}
If in addition $u$ is stable, then in fact $u$ must be identically zero. 
\end{thm}
Note  that in the absence of stability it is expected that the only nonnegative bounded solution of (\ref{main}) must be zero  for the subcritical exponents $1<p<p_S(n,a)$ where $a \ge0$.  To our knowledge not much is known about the classification of solutions when $a\neq 0$ even for the standard case $s=1$. For the case Laplacian operator $s=1$,  Phan-Souplet in \cite{ps} proved that the only nonnegative bounded solution of (\ref{main}) in three dimensions  must be zero  for the case of $1<p<p_S(n,a)$ and $a>-2$. Some partial results are given in \cite{gs}.

\section{The monotonicity formula}
Here is the monotonicity formula for the case of $0<s<1$.
\begin{thm}\label{mono1}
Suppose that $0<s<1$.  Let $ u_e\in C^2(\r)\cap C(\super\r)$ be  a solution of \eqref{main}  such that $y^{1-2s}\partial_{y} u_e\in C(\super\r)$. For $x_{0}\in\br$, $\lambda>0$,  let
\begin{eqnarray}
\nonumber E( u_e,\lambda) &:=&  \lambda^{\frac{2s(p+1)+2a}{p-1}-n}\left(\frac12\int_{\r\cap B_\lambda} y^{1-2s}\vert\nabla u_e \vert^2\;dx\,dy - \frac{\kappa_{s}}{p+1}\int_{\br\cap B_\lambda} |x|^a \vert  u_e \vert^{p+1}\;dx\right)\\
&&+ \lambda^{\frac{2s(p+1)+2a}{p-1}-n-1}\frac{s +\frac{a}{2}  }{p+1}\int_{\partial B_\lambda \cap\r}y^{1-2s} u_e^2\;d\sigma.
\end{eqnarray}
Then, $E$ is a nondecreasing function of $\lambda$. Furthermore,
\begin{equation} 
\frac{dE}{d\lambda} = \lambda^{\frac{2s(p+1)+a}{p-1}-n+1}\int_{\partial B(x_{0},\lambda)\cap\r}y^{1-2s}\left(\frac{\partial  u_e}{\partial r}+\frac{2s+a}{p-1}\frac { u_e}r\right)^2\;d\sigma. 
   \end{equation}
\end{thm}
\begin{proof}   Let
\begin{equation} \label{I}
I(u_e,\lambda) = \lambda^{2s\frac{p+1}{p-1}-n}\left(\int_{\r\cap B_\lambda} y^{1-2s}\frac{\vert\nabla u_e\vert^2}{2}dx\,dy -\frac{\kappa_{s}}{p+1} \int_{\br\cap B_\lambda} |x|^a \vert u_e \vert^{p+1}dx\right). 
\end{equation}
Now for $X\in\r$, define
\begin{equation}\label{ulambdas1}
u^\lambda_e(X) : = \lambda^{\frac{2s+a}{p-1}} u_e(\lambda X).
\end{equation}
Then, $u^\lambda_e$ solves \eqref{maine} and in addition
\begin{equation}\label{t3}
I(u_e,\lambda) = I(u^\lambda_e,1).
\end{equation}
Taking partial derivatives we get
\begin{equation} \label{lambdader}
\lambda \partial_\lambda u^\lambda_e = \frac{2s+a}{p-1} u^\lambda_e + r\partial_r u^\lambda_e.
\end{equation}
Differentiating the operator \eqref{I} with respect to $\lambda$, we find
$$
\partial_\lambda I(u_e,\lambda) = \int_{\r\cap B_1}y^{1-2s}\nabla u^\lambda_e\cdot\nabla\partial_\lambda u^\lambda_e dx\,dy -\kappa_{s}\int_{\br\cap B_1}|x|^a \vert u^\lambda_e \vert^{p-1}\partial_\lambda u^\lambda_e dx.
$$
Integrating by parts and then using \eqref{lambdader},
\begin{align*}
\partial_\lambda I(u_e,\lambda)  &= \int_{\partial B_1\cap\r}y^{1-2s} \partial_r  u^\lambda_e \partial_\lambda  u^\lambda_e d\sigma\\
&= \lambda \int_{\partial B_1\cap\r}y^{1-2s} (\partial_\lambda  u^\lambda_e)^2 d\sigma - \frac{2s+a}{p-1}\int_{\partial B_1\cap\r}y^{1-2s}  u^\lambda_e \partial_\lambda  u^\lambda_e  d\sigma\\
&= \lambda \int_{\partial B_1\cap\r} y^{1-2s} (\partial_\lambda  u^\lambda_e)^2 d\sigma - \frac{s+\frac{a}{2}}{p-1}\partial_\lambda \left(\int_{\partial B_1\cap\r}y^{1-2s} ( u^\lambda_e)^2 \;d\sigma\right) . 
\end{align*}
This implies that 
\begin{equation}
\partial_\lambda\left[ I(u_e,\lambda)+ \frac{s+\frac{a}{2}}{p-1}\int_{\partial B_1\cap\r}y^{1-2s} ( u^\lambda_e)^2 \;d\sigma\right] = \lambda \int_{\partial B_1\cap\r} y^{1-2s} (\partial_\lambda  u^\lambda_e)^2 d\sigma . 
\end{equation}
Applying the scaling (\ref{ulambdas1}) completes the proof. 
\end{proof}
We now consider the case of $1<s<2$ and $a>0$ and we drive a monotonicity formula for this case.  Note that when $a=0$ a monotonicity formula is given for $s=2$ and $1<s<2$ by Davila-Dupaigne-Wang-Wei in \cite{ddww} and Fazly-Wei in \cite{fw}, respectively. We first define the energy functional
\begin{eqnarray*}\label{energy}
E(u_e,r)& :=& r^{2s\frac{p+1}{p-1}-n} \left(   \int_{  \mathbb{R}^{n+1}_{+}\cap B_r} \frac{1}{2} y^{3-2s}|\Delta_b u_e|^2-  \frac{C_{n,s}}{p+1} \int_{  \partial\mathbb{R}^{n+1}_{+}\cap B_r} |x|^a u_e^{p+1}   \right)\\
&&-\frac{s+\frac{a}{2}}{p-1}\left(  \frac{p+2s+a-1}{p-1} -n-b \right)  r^{-3+2s+\frac{4s+2a}{p-1}-n}  \int_{  \mathbb{R}^{n+1}_{+}\cap \partial B_r} y^{3-2s} u_e^2 \\
&&-\frac{s+\frac{a}{2}}{p-1}\left(  \frac{p+2s+a-1}{p-1} -n-b \right) \frac{d}{dr} \left[ r^{\frac{4s+2a}{p-1}+2s-2-n}  \int_{  \mathbb{R}^{n+1}_{+}\cap \partial B_r} y^{3-2s} u_e^2 \right]\\
&&+ \frac{1}{2} r^3   \frac{d}{dr} \left[ r^{\frac{4s+2a}{p-1}+2s-3-n}  \int_{  \mathbb{R}^{n+1}_{+}\cap \partial B_r} y^{3-2s} \left(  \frac{2s+a}{p-1} r^{-1} u+ \frac{\partial u_e}{\partial r}\right)^2 \right]\\
&&+ \frac{1}{2}    \frac{d}{dr} \left[ r^{\frac{2s(p+1)+2a}{p-1}-n}  \int_{  \mathbb{R}^{n+1}_{+}\cap \partial B_r} y^{3-2s}\left(  | \nabla u_e|^2 - \left|\frac{\partial u_e}{\partial r}\right|^2 \right) \right]\\
&&+ \frac{1}{2}    r^{\frac{2s(p+1)+2a}{p-1}-n-1}  \int_{  \mathbb{R}^{n+1}_{+}\cap \partial B_r} y^{3-2s} \left(  | \nabla u_e|^2 - \left|\frac{\partial u_e}{\partial r}\right|^2 \right). 
\end{eqnarray*}
For the above energy functional we have provide the following monotonicity formula. 
\begin{thm}\label{mono}
Assume that $n>\frac{p+4s+2a-1}{p+2s+a-1}+ \frac{2s+a}{p-1}-b$. Then, $E(u_e,\lambda)$ is a nondecreasing function of $\lambda>0$. Furthermore,
\begin{equation}
\frac{dE(\lambda,u_e)}{d\lambda} \ge C(n,s,p) \  \lambda^{\frac{4s+2a}{p-1}+2s-2-n}    \int_{  \mathbb{R}^{n+1}_{+}\cap \partial B_\lambda}  y^{3-2s}\left(  \frac{2s+a}{p-1} r^{-1} u+ \frac{\partial u_e}{\partial r}\right)^2, 
\end{equation}
where $C(n,s,p)$ is independent from $\lambda$.
\end{thm}

\noindent{\bf Proof:}  Set 
\begin{equation}
\bar E(u_e,\lambda):= \lambda^{\frac{2s(p+1)+2a}{p-1}-n} \left(   \int_{  \mathbb{R}^{n+1}_{+}\cap B_\lambda} \frac{1}{2} y^{b} |\Delta_b u_e|^2 dx dy -  \frac{C_{n,s}}{p+1} \int_{  \partial\mathbb{R}^{n+1}_{+}\cap B_\lambda} |x|^a u_e^{p+1}   \right) . 
\end{equation}
 Define $v_e:=\Delta_b u_e$, $u_e^\lambda(X):=\lambda^{\frac{2s+a}{p-1}} u_e(\lambda X)$, and $v_e^\lambda(X):=\lambda^{\frac{2s+a}{p-1}+2} v_e(\lambda X)$ where $X=(x,y)\in\mathbb{R}^{n+1}_+$.  Therefore, $\Delta_b u_e^\lambda(X)=v_e^\lambda(X)$ and
 \begin{eqnarray}\label{mainex}
 \left\{ \begin{array}{lcl}
\hfill \Delta_b v^\lambda_e&=& 0   \ \ \text{in}\ \ \mathbb{R}^{n+1}_{+},\\
\hfill  \lim_{y\to 0} y^{b}\partial_y{u^\lambda_e}&=& 0   \ \ \text{in}\ \ \partial\mathbb{R}^{n+1}_{+},\\
\hfill \lim_{y\to 0} y^{b} \partial_y v_e^\lambda &=& C_{n,s} |x|^a {(u^\lambda_e)}^p  \ \ \text{in}\ \ \mathbb{R}^{n}. 
\end{array}\right.
\end{eqnarray}
In addition,  differentiating with respect to $\lambda$ we have
 \begin{equation}\label{uvl}
 \Delta_b \frac{du_e^\lambda}{d\lambda}=\frac{dv_e^\lambda}{d\lambda}.
 \end{equation}
Note that $$\bar E(u_e,\lambda)=\bar E(u_e^\lambda,1)= \int_{  \mathbb{R}^{n+1}_{+}\cap B_1} \frac{1}{2} y^b  (v_e^\lambda)^2 dx dy -  \frac{C_{n,s}}{p+1} \int_{  \partial\mathbb{R}^{n+1}_{+}\cap B_1} |x|^a |u_e^\lambda|^{p+1}  .   $$
Taking derivate of the energy with respect to $\lambda$, we have
 \begin{eqnarray}\label{energy}
\frac{d\bar E(u_e^\lambda,1)}{d\lambda}= \int_{  \mathbb{R}^{n+1}_{+}\cap B_1} y^b v_e^\lambda \frac{dv_e^\lambda}{d\lambda}\  dx dy -  C_{n,s} \int_{  \partial\mathbb{R}^{n+1}_{+}\cap B_1}  |x|^a |u_e^\lambda|^{p}  \frac{du_e^\lambda}{d\lambda} . 
\end{eqnarray}
Using (\ref{mainex}) we end up with
\begin{eqnarray}\label{energy}
\frac{d\bar E(u_e^\lambda,1)}{d\lambda}= \int_{  \mathbb{R}^{n+1}_{+}\cap B_1}  y^b v_e^\lambda \frac{dv_e^\lambda}{d\lambda}\  dx dy -  \int_{  \partial\mathbb{R}^{n+1}_{+}\cap B_1}  \lim_{y\to 0} y^{b}  \partial_y v_e^\lambda   \frac{du_e^\lambda}{d\lambda} . 
\end{eqnarray}
From (\ref{uvl}) and by  integration by parts we have
 \begin{eqnarray*}
 \int_{  \mathbb{R}^{n+1}_{+}\cap B_1} y^b v_e^\lambda \frac{dv_e^\lambda}{d\lambda}
&=&  \int_{  \mathbb{R}^{n+1}_{+}\cap B_1} y^b \Delta_b u_e^\lambda \Delta_b \frac{du_e^\lambda}{d\lambda} \\&=&  - \int_{  \mathbb{R}^{n+1}_{+}\cap B_1} \nabla \Delta_b u^\lambda_e \cdot \nabla \left(\frac{du^\lambda_e}{d \lambda}\right) y^b + \int_{ \partial( \mathbb{R}^{n+1}_{+}\cap B_1)} \Delta_b u_e^\lambda y^b \partial_{\nu} \left(   \frac{du^\lambda_e}{d \lambda}  \right) . 
\end{eqnarray*}  
Note that
 \begin{eqnarray*}
-\int_{  \mathbb{R}^{n+1}_{+}\cap B_1} \nabla \Delta_b u_e\cdot \nabla \frac{du^\lambda_e}{d \lambda} y^b &=& \int_{  \mathbb{R}^{n+1}_{+}\cap B_1} \div( \nabla\Delta_b u^\lambda_e y^b) \frac{du^\lambda_e}{d \lambda} -  \int_{ \partial( \mathbb{R}^{n+1}_{+}\cap B_1)} y^b   \partial_\nu (\Delta_b u^\lambda_e )  \frac{du^\lambda_e}{d \lambda}
\\& =&\int_{  \mathbb{R}^{n+1}_{+}\cap B_1} y^b  \Delta_b^2 u^\lambda_e  \frac{du^\lambda_e}{d \lambda} -  \int_{ \partial( \mathbb{R}^{n+1}_{+}\cap B_1)} y^b   \partial_\nu (\Delta_b u^\lambda_e )  \frac{du^\lambda_e}{d \lambda}
\\& =&-  \int_{ \partial( \mathbb{R}^{n+1}_{+}\cap B_1)} y^b   \partial_\nu (\Delta_b u^\lambda_e )  \frac{du^\lambda_e}{d \lambda} . 
\end{eqnarray*}
Therefore,
 \begin{eqnarray*}
 \int_{  \mathbb{R}^{n+1}_{+}\cap B_1} y^b v_e^\lambda \frac{dv_e^\lambda}{d\lambda}
&=& \int_{ \partial( \mathbb{R}^{n+1}_{+}\cap B_1)} \Delta_b u_e^\lambda y^b \partial_{\nu} \left(   \frac{du^\lambda_e}{d \lambda}  \right) -  \int_{ \partial( \mathbb{R}^{n+1}_{+}\cap B_1)} y^b   \partial_\nu (\Delta_b u^\lambda_e )  \frac{du^\lambda_e}{d \lambda} . 
\end{eqnarray*}
Boundary of $\mathbb{R}^{n+1}_{+}\cap B_1$ consists of   $\partial\mathbb{R}^{n+1}_{+}\cap B_1$ and  $\mathbb{R}^{n+1}_{+}\cap \partial B_1$. Therefore,
\begin{eqnarray*}
 \int_{  \mathbb{R}^{n+1}_{+}\cap B_1} y^b v_e^\lambda \frac{dv_e^\lambda}{d\lambda}
&=&  \int_{  \partial\mathbb{R}^{n+1}_{+}\cap B_1} - v_e^\lambda \lim_{y\to 0}  y^b \partial_y\left (\frac{du_e^\lambda}{d\lambda}\right) +   \lim_{y\to 0} y^b \partial_y v_e^\lambda \frac{du_e^\lambda}{d\lambda} \\&&+ \int_{  \mathbb{R}^{n+1}_{+}\cap \partial B_1} y^b v_e^\lambda \partial_r \left (\frac{du_e^\lambda}{d\lambda}\right) -  y^b \partial_r v_e^\lambda \frac{du_e^\lambda}{d\lambda} , 
\end{eqnarray*}
where $r=|X|$, $X=(x,y)\in \mathbb{R}^{n+1}_+$ and $\partial_r=\nabla\cdot \frac{X}{r}$ is the corresponding radial derivative.   Note that the first integral in the right-hand side vanishes since $\partial_y\left (\frac{du_e^\lambda}{d\lambda}\right)=0$ on $ \partial\mathbb{R}^{n+1}_{+}$. From (\ref{energy}) we obtain
 \begin{eqnarray}\label{energy2}
\frac{d\bar E(u_e^\lambda,1)}{d\lambda}= \int_{  \mathbb{R}^{n+1}_{+}\cap \partial B_1}  y^b\left ( v_e^\lambda \partial_r \left (\frac{du_e^\lambda}{d\lambda}\right) -  \partial_r \left(v_e^\lambda\right) \frac{du_e^\lambda}{d\lambda} \right)  . 
\end{eqnarray}
Now note that from the definition of $u_e^\lambda$ and $v_e^\lambda$ and by differentiating in $\lambda $ we get the following for $X\in\mathbb{R}^{n+1}_+$
\begin{eqnarray}\label{ulambda}
\frac{du_e^\lambda(X)}{d\lambda}&=&\frac{1}{\lambda} \left(  \frac{2s+a}{p-1} u_e^\lambda(X) +r \partial_r u_e^\lambda(X) \right) , 
\\    
\label{vlambda} 
\frac{dv_e^\lambda(X)}{d\lambda}&=&\frac{1}{\lambda} \left(  \frac{2(p+s-1)+a}{p-1} v_e^\lambda(X) +r \partial_r v_e^\lambda(X) \right)  . 
\end{eqnarray}
Therefore, differentiating with respect to $\lambda $ we get
\begin{eqnarray*}
\lambda \frac{d^2u_e^\lambda(X)}{d\lambda^2} + \frac{du_e^\lambda(X)}{d\lambda}=  \frac{2s+a}{p-1} \frac{du_e^\lambda(X)}{d\lambda} +r \partial_r \frac{du_e^\lambda(X)}{d\lambda} . 
\end{eqnarray*}
 From this for all $X\in\mathbb{R}^{n+1}_+\cap \partial B_1$ we get 
\begin{eqnarray}
\label{u1lambda}
 \partial_r \left(u_e^\lambda(X)\right)
&=& \lambda \frac{du_e^\lambda(X)}{d\lambda} -\frac{2s+a}{p-1} u_e^\lambda(X) , 
\\\label{u2lambda}
 \partial_r \left(\frac{du_e^\lambda(X)}{d\lambda}\right)
&= &\lambda \frac{d^2u_e^\lambda(X)}{d\lambda^2} +\frac{p-1-2s-a}{p-1} \frac{du_e^\lambda(X)} {d\lambda} , 
\\\label{v2lambda}
\partial_r \left( v_e^\lambda(X)\right) &=& \lambda \frac{dv_e^\lambda(X)}{d\lambda}- \frac{2(p+s-1)+a}{p-1} v_e^\lambda(X) . 
\end{eqnarray}
Substituting (\ref{u2lambda}) and (\ref{v2lambda}) in (\ref{energy2}) we get
 \begin{eqnarray}\label{energyder}
\nonumber \frac{d\bar E(u_e^\lambda,1)}{d\lambda}&=&
\int_{  \mathbb{R}^{n+1}_{+}\cap \partial B_1} y^b v_e^\lambda  \left (   \lambda \frac{d^2u_e^\lambda}{d\lambda^2} +\frac{p-1-2s-a}{p-1} \frac{du_e^\lambda}{d\lambda}\right) 
\\&& \nonumber-   \int_{  \mathbb{R}^{n+1}_{+}\cap \partial B_1} y^b \left( \lambda \frac{dv_e^\lambda}{d\lambda}- \frac{2(p+s-1)+a}{p-1} v_e^\lambda \right) \frac{du_e^\lambda}{d\lambda}
\\&=&  \int_{  \mathbb{R}^{n+1}_{+}\cap \partial B_1} y^b \left( \lambda  v_e^\lambda    \frac{d^2u_e^\lambda}{d\lambda^2} +3  v_e^\lambda \frac{du_e^\lambda}{d\lambda}    -    \lambda \frac{dv_e^\lambda}{d\lambda} \frac{du_e^\lambda}{d\lambda} \right) .
\end{eqnarray}
Taking derivative of (\ref{ulambda}) in $r$ we get
\begin{equation}
 r \frac{\partial^2 u_e^\lambda}{\partial r^2}+ \frac{\partial u_e^\lambda}{\partial r}= \lambda \frac{\partial}{\partial r}\left(\frac{du_e^\lambda}{d\lambda} \right) - \frac{2s+a}{p-1}  \frac{\partial u_e^\lambda}{\partial r} . 
\end{equation}
From this and (\ref{u2lambda}) for all $X\in\mathbb{R}^{n+1}_+\cap \partial B_1$ we have
 \begin{eqnarray}\label{2ru}
\nonumber \frac{\partial^2 u_e^\lambda}{\partial r^2} &=& \lambda \frac{\partial}{\partial r}\left(\frac{du_e^\lambda}{d\lambda} \right) -  \frac{p+2s+a-1}{p-1}  \frac{\partial u_e^\lambda}{\partial r} \\&=& \nonumber \lambda \left(  \lambda \frac{d^2u_e^\lambda}{d\lambda^2} +\frac{p-2s-1-a}{p-1} \frac{du_e^\lambda}{d\lambda}  \right) -   \frac{p+2s+a-1}{p-1}  \left(  \lambda \frac{du_e^\lambda}{d\lambda} -\frac{2s+a}{p-1} u_e^\lambda \right)
\\&=&
 \lambda^2 \frac{d^2u_e^\lambda}{d\lambda^2} - \frac{4s+2a}{p-1}   \lambda \frac{du_e^\lambda}{d\lambda} +\frac{(2s+a)(p+2s+a-1)}{(p-1)^2} u_e^\lambda
\end{eqnarray}
Note that using the definition of the operator $\Delta_b$ and $v$ we have 
 \begin{equation}
v_e^\lambda= \Delta_b u^\lambda_e = y^{-b} \div(y^b \nabla u^\lambda_e),
\end{equation}
and on $\mathbb{R}^{n+1}_+\cap \partial B_1$ we have 
 \begin{equation}
 \div(y^b \nabla u_e^\lambda )=(u_{rr} +(n+b)u_r )\theta_1^b +\div_{\mathcal S^n} (\theta_1^b \nabla_{S^n} u_e^\lambda), 
 \end{equation}
where $\theta_1=\frac {y}{r}$. From the above, (\ref{u1lambda}) and (\ref{2ru}) we obtain 
\begin{eqnarray*}
v_e^\lambda &=& \lambda^2 \frac{d^2u_e^\lambda}{d\lambda^2} +    \lambda \frac{du_e^\lambda}{d\lambda} (n+b-\frac{4s+2a}{p-1} )+ u_e^\lambda  (\frac{2s+a}{p-1})(\frac{p+2s+a-1}{p-1}-n-b) + \theta_1^{-b}\div_{\mathcal S^n} (\theta_1^b \nabla_{S^n} u_e^\lambda). 
\end{eqnarray*}
From this and (\ref{energyder}) we get  
 \begin{eqnarray}
\frac{d\bar E(u_e^\lambda,1)}{d\lambda}&=& \int_{  \mathbb{R}^{n+1}_{+}\cap \partial B_1} \theta_1^b \lambda  \left(    \lambda^2 \frac{d^2u_e^\lambda}{d\lambda^2} +\alpha \lambda \frac{du_e^\lambda}{d\lambda} + \beta u_e^\lambda     \right)   \frac{d^2u_e^\lambda}{d\lambda^2}
\\&&+  \int_{  \mathbb{R}^{n+1}_{+}\cap \partial B_1}   \theta_1^b 3\left(    \lambda^2 \frac{d^2u_e^\lambda}{d\lambda^2} +\alpha \lambda \frac{du_e^\lambda}{d\lambda} + \beta u_e^\lambda     \right)          \frac{du_e^\lambda}{d\lambda}
\\&&- \int_{  \mathbb{R}^{n+1}_{+}\cap \partial B_1}     \theta_1^b  \lambda  \frac{du_e^\lambda}{d\lambda}  \frac{d}{d\lambda}  \left(    \lambda^2 \frac{d^2u_e^\lambda}{d\lambda^2} +\alpha \lambda \frac{du_e^\lambda}{d\lambda} + \beta u_e^\lambda     \right)
\\&& +\int_{  \mathbb{R}^{n+1}_{+}\cap \partial B_1} \theta_1^b \lambda   \frac{d^2u_e^\lambda}{d\lambda^2} \theta_1^{-b}\div_{\mathcal S^n} (\theta_1^b \nabla_{S^n} u_e^\lambda)
 \\&& +\int_{  \mathbb{R}^{n+1}_{+}\cap \partial B_1}  3  \theta_1^b \ \frac{du_e^\lambda}{d\lambda}  \theta_1^{-b}\div_{\mathcal S^n} (\theta_1^b \nabla_{S^n} u_e^\lambda)  \\&& -  \int_{  \mathbb{R}^{n+1}_{+}\cap \partial B_1} \theta_1^b  \lambda \frac{d}{d\lambda}\left(  \theta_1^{-b}\div_{\mathcal S^n} (\theta_1^b \nabla_{S^n} u_e^\lambda)  \right) \frac{du_e^\lambda}{d\lambda}
\end{eqnarray}
where $\alpha:=n + b- \frac{4s+2a}{p-1}$ and $\beta:=\frac{2s+a}{p-1}\left(  \frac{p+2s+a-1}{p-1} -n-b \right)$. Simplifying the integrals we get
\begin{eqnarray}
\label{de}\frac{d\bar E(u_e^\lambda,1)}{d\lambda}&=&  \int_{  \mathbb{R}^{n+1}_{+}\cap \partial B_1}  \theta_1^b \left( 2 \lambda^3    \left(   \frac{d^2u_e^\lambda}{d\lambda^2}\right)^2 + 4 \lambda^2  \frac{d^2u_e^\lambda}{d\lambda^2}  \frac{du_e^\lambda}{d\lambda} +2(\alpha-\beta) \lambda \left(    \frac{du_e^\lambda}{d\lambda}\right)^2 \right)
 \\&& \nonumber + \int_{  \mathbb{R}^{n+1}_{+}\cap \partial B_1}  \theta_1^b \left( \frac{\beta}{2} \frac{d^2}{d\lambda^2} \left(   \lambda (u_e^\lambda)^2    \right) -\frac{1}{2} \frac{d}{d\lambda}  \left(   \lambda^3      \frac{d}{d\lambda} \left(   \frac{d  u_e^\lambda }{d\lambda} \right)^2     \right) +\frac{\beta}{2} \frac{d}{d\lambda}(u_e^\lambda)^2 \right)
\\&& \nonumber+\int_{  \mathbb{R}^{n+1}_{+}\cap \partial B_1}   \lambda \frac{d^2u_e^\lambda}{d\lambda^2} \div_{\mathcal S^n} (\theta_1^b \nabla_{S^n} u_e^\lambda)  +3  \div_{\mathcal S^n} (\theta_1^b \nabla_{S^n} u_e^\lambda)    \frac{du_e^\lambda}{d\lambda}   
\\&&  \nonumber -  \int_{  \mathbb{R}^{n+1}_{+}\cap \partial B_1}   \lambda \frac{d}{d\lambda}\left(  \div_{\mathcal S^n} (\theta_1^b \nabla_{S^n} u_e^\lambda)   \right) \frac{du_e^\lambda}{d\lambda} . 
\end{eqnarray}
Note that from the assumptions we have  $\alpha-\beta-1>0$, therefore the first term in the right-hand side of (\ref{de}) is positive that is
\begin{eqnarray}
&&2 \lambda^3    \left(   \frac{d^2u_e^\lambda}{d\lambda^2}\right)^2 + 4 \lambda^2  \frac{d^2u_e^\lambda}{d\lambda^2}  \frac{du_e^\lambda}{d\lambda} +2(\alpha-\beta) \lambda \left(    \frac{du_e^\lambda}{d\lambda}\right)^2
\\&=& 2 \lambda \left(  \lambda  \frac{d^2u_e^\lambda}{d\lambda^2}  +   \frac{du_e^\lambda}{d\lambda}  \right)^2 +2(\alpha-\beta-1) \lambda \left(    \frac{du_e^\lambda}{d\lambda}\right)^2 >0 . 
\end{eqnarray}
From this we obtain 
\begin{eqnarray*}
\label{de1}\frac{d\bar E(u_e^\lambda,1)}{d\lambda}& \ge &  \int_{  \mathbb{R}^{n+1}_{+}\cap \partial B_1}  \theta_1^b\left( \frac{\beta}{2} \frac{d^2}{d\lambda^2} \left(   \lambda (u_e^\lambda)^2    \right) -\frac{1}{2} \frac{d}{d\lambda}  \left(   \lambda^3      \frac{d}{d\lambda} \left(   \frac{d  u_e^\lambda }{d\lambda} \right)^2     \right) +\frac{\beta}{2} \frac{d}{d\lambda}(u_e^\lambda)^2 \right)
\\&& \nonumber+\int_{  \mathbb{R}^{n+1}_{+}\cap \partial B_1}   \lambda \frac{d^2u_e^\lambda}{d\lambda^2} \div_{\mathcal S^n} (\theta_1^b \nabla_{S^n} u_e^\lambda)  +3  \div_{\mathcal S^n} (\theta_1^b \nabla_{S^n} u_e^\lambda)    \frac{du_e^\lambda}{d\lambda}    -    \lambda \frac{d}{d\lambda}\left(  \div_{\mathcal S^n} (\theta_1^b \nabla_{S^n} u_e^\lambda)   \right) \frac{du_e^\lambda}{d\lambda} \\&=:& R_1+R_2.
\end{eqnarray*}
Note that the three terms appeared in $R_1$ are of the following form
\begin{eqnarray*}
 \int_{  \mathbb{R}^{n+1}_{+}\cap \partial B_1}
\theta_1^b \frac{d^2}{d\lambda^2} \left(   \lambda (u_e^\lambda)^2    \right) &=&  \frac{d^2}{d\lambda^2} \left(   \lambda^{ \frac{4s+2a}{p-1}+2(s-1)-n } \int_{  \mathbb{R}^{n+1}_{+}\cap \partial B_\lambda}
y^b  u_e^2    \right) \\
 \int_{  \mathbb{R}^{n+1}_{+}\cap \partial B_1}
\theta_1^b \frac{d}{d\lambda}  \left[  \lambda^3      \frac{d}{d\lambda} \left(   \frac{d  u_e^\lambda }{d\lambda} \right)^2     \right] &=&  \frac{d}{d\lambda}  \left[   \lambda^3      \frac{d}{d\lambda} \left(  \lambda^{ \frac{4s+2a}{p-1}+2s-3-n }    \int_{  \mathbb{R}^{n+1}_{+}\cap \partial B_\lambda} y^b \left[  \frac{2s+a}{p-1} \lambda^{-1} u_e + \frac{\partial u_e}{\partial r}      \right]^2     \right)   \right]
 \\
  \int_{  \mathbb{R}^{n+1}_{+}\cap \partial B_1}  y^b \frac{d}{d\lambda}(u_e^\lambda)^2 &=& \frac{d}{d\lambda} \left( \lambda^{  2s-3+ \frac{4s+2a}{p-1}-n }  \int_{  \mathbb{R}^{n+1}_{+}\cap \partial B_\lambda}      y^b u_e ^2\right) . 
\end{eqnarray*}
We now apply integration by parts to simplify the terms appeared in $R_2$.
\begin{eqnarray*}
 R_2 &=&  \int_{  \mathbb{R}^{n+1}_{+}\cap \partial B_1}   \lambda \frac{d^2u_e^\lambda}{d\lambda^2} \div_{\mathcal S^n} (\theta_1^b \nabla_{S^n} u_e^\lambda)  +3  \div_{\mathcal S^n} (\theta_1^b \nabla_{S^n} u_e^\lambda)    \frac{du_e^\lambda}{d\lambda}    -    \lambda \frac{d}{d\lambda}\left(  \div_{\mathcal S^n} (\theta_1^b \nabla_{S^n} u_e^\lambda)   \right) \frac{du_e^\lambda}{d\lambda} \\
 &=&  \int_{  \mathbb{R}^{n+1}_{+}\cap \partial B_1} - \theta_1^b  \lambda  \nabla_{\mathcal S^n} u_e^\lambda \cdot   \nabla_{\mathcal S^n}  \frac{d^2u_e^\lambda}{d\lambda^2} - 3 \theta_1^b   \nabla_{\mathcal S^n}  u_e^\lambda \cdot    \nabla_{\mathcal S^n}  \frac{du_e^\lambda}{d\lambda}  +  \theta_1^b\lambda \left|     \nabla_{\mathcal S^n}  \frac{du_e^\lambda}{d\lambda}   \right|^2\\
&=& - \frac{\lambda}{2}  \frac{d^2}{d\lambda^2} \left(  \int_{  \mathbb{R}^{n+1}_{+}\cap \partial B_1} \theta_1^b  |\nabla_\theta u_e^\lambda |^2   \right) -\frac{3}{2} \frac{d}{d\lambda} \left(  \int_{  \mathbb{R}^{n+1}_{+}\cap \partial B_1}  \theta_1^b |\nabla_\theta u_e^\lambda |^2   \right)+2\lambda \int_{  \mathbb{R}^{n+1}_{+}\cap \partial B_1}  \theta_1^b \left|    \nabla_\theta \frac{du_e^\lambda}{d\lambda}   \right|^2
\\&=& - \frac{1}{2}  \frac{d^2}{d\lambda^2} \left(   \lambda      \int_{  \mathbb{R}^{n+1}_{+}\cap \partial B_1} \theta_1^b  |\nabla_\theta u_e^\lambda |^2   \right) -\frac{1}{2} \frac{d}{d\lambda} \left(  \int_{  \mathbb{R}^{n+1}_{+}\cap \partial B_1} \theta_1^b  |\nabla_\theta u_e^\lambda |^2   \right)+2\lambda \int_{  \mathbb{R}^{n+1}_{+}\cap \partial B_1} \theta_1^b \left|    \nabla_\theta \frac{du_e^\lambda}{d\lambda}   \right|^2
\\&\ge& - \frac{1}{2}  \frac{d^2}{d\lambda^2} \left(   \lambda      \int_{  \mathbb{R}^{n+1}_{+}\cap \partial B_1}   \theta_1^b |\nabla_\theta u_e^\lambda |^2   \right) -\frac{1}{2} \frac{d}{d\lambda} \left(  \int_{  \mathbb{R}^{n+1}_{+}\cap \partial B_1}  \theta_1^b |\nabla_\theta u_e^\lambda |^2   \right) . 
 \end{eqnarray*}
 Note that the two terms that appear as lower bound for $R_3$ are of the form
  \begin{eqnarray}
 \frac{d^2}{d\lambda^2} \left(   \lambda      \int_{  \mathbb{R}^{n+1}_{+}\cap \partial B_1} \theta_1^b  |\nabla_\theta u_e^\lambda |^2   \right) &=&  \frac{d^2}{d\lambda^2} \left[  \lambda^{\frac{2s(p+1)+2a}{p-1}-n}   \int_{  \mathbb{R}^{n+1}_{+}\cap \partial B_\lambda} y^b  \left(  |\nabla  u|^2-\left|     \frac{\partial u}{\partial r}\right|^2  \right)
  \right]
  \\
   \frac{d}{d\lambda} \left(  \int_{  \mathbb{R}^{n+1}_{+}\cap \partial B_1}  \theta_1^b |\nabla_\theta u_e^\lambda |^2   \right) &=&  \frac{d}{d\lambda} \left[  \lambda^{\frac{2s(p+1)+2a}{p-1}-n-1}     \int_{  \mathbb{R}^{n+1}_{+}\cap \partial B_\lambda} y^b \left(  |\nabla  u|^2-\left|     \frac{\partial u}{\partial r}\right|^2  \right)      \right] . 
  \end{eqnarray}

   \hfill $ \Box$

 \begin{remark}
 It is straightforward to show that $n>\frac{2s(p+1)+2a}{p-1} $ implies $n>\frac{p+4s+2a-1}{p+2s+a-1}+ \frac{2s+a}{p-1}-b$.
 \end{remark}

 \section{Homogeneous solutions}

 \begin{thm}\label{homog}
 Suppose that $u=r^{-\frac{2s+a}{p-1}} \psi(\theta)$ is a stable solution of (\ref{main}). Then $\psi$ must be identically zero,  provided
 $p>\frac{n+2s+2a}{n-2s}$ and
 \begin{equation} \label{cond}
p \frac{\Gamma(\frac n2-\frac{s+\frac{a}{2}}{p-1}) \Gamma(s+\frac{s+\frac{a}{2}}{p-1})}{\Gamma(\frac{s+\frac{a}{2}}{p-1}) \Gamma(\frac{n-2s}{2} - \frac{s+\frac{a}{2}}{p-1})} > \frac{\Gamma(\frac{n+2s}{4})^2}{\Gamma(\frac{n-2s}{4})^2}. 
\end{equation}
 \end{thm}
  \begin{proof}
 Since $u$ satisfies (\ref{main}), the function $\psi$ satisfies
 \begin{eqnarray*}
 |x|^{a} |x|^{-\frac{2ps+ap}{p-1}}  \psi^p(\theta)
&=& \int \frac{  |x|^{-\frac{2s+a}{p-1}} \psi(\theta) -|y|^{-\frac{2s+a}{p-1}} \psi(\sigma) }{ |x-y|^{n+2s}} dy \\&=&
  \int \frac{  |x|^{-\frac{2s+a}{p-1}} \psi(\theta) -r^{-\frac{2s+a}{p-1}} t^{-\frac{2s+a}{p-1}} \psi(\sigma) }{ (t^2+1- 2t <\theta,\sigma>)^{\frac{n+2s}{2}} |x|^{n+2s}} |x|^n t^{n-1} dt d\sigma \text{ \ where \ } |y|=rt
\\ &=&  |x|^{-\frac{2ps+a}{p-1}}  [  \int \frac{  \psi(\theta) - t^{-\frac{2s+a}{p-1}} \psi(\theta) }{ (t^2+1- 2t <\theta,\sigma>)^{\frac{n+2s}{2}} }  t^{n-1} dt d\sigma \\&&+ \int \frac{  t^{-\frac{2s+a}{p-1}} (\psi(\theta) -  \psi(\sigma) }{ (t^2+1- 2t <\theta,\sigma>)^{\frac{n+2s}{2}} }  t^{n-1} dt d\sigma] . 
\end{eqnarray*}
 We now drop $|x|^{-\frac{2ps+a}{p-1}}$ and get
 \begin{equation}\label{Ans}
 \psi(\theta) A_{n,s,a}(\theta) + \int_{\mathbb S^{n-1}} K_{\frac{2s+a}{p-1}} (<\theta,\sigma>) (\psi(\theta)-\psi(\sigma)) d \sigma= \psi^p(\theta) , 
 \end{equation}
 where 
  \begin{equation} 
  A_{n,s,a}:=\int_0^\infty \int_{\mathbb S^{n-1}} \frac{1-t^{  -\frac{2s+a}{p-1} }}{    (t^2+1- 2t <\theta,\sigma>)^{\frac{n+2s}{2}} } t^{n-1} d\sigma dt , 
   \end{equation}
 and 
  \begin{equation} 
  K_{\frac{2s+a}{p-1}}(<\theta,\sigma>):=\int_0^\infty \frac{t^{  n-1-\frac{2s}{p-1} }}{    (t^2+1- 2t <\theta,\sigma>)^{\frac{n+2s}{2}} }  dt   . 
  \end{equation}
 Note that
  \begin{eqnarray}
K_{\frac{2s+a}{p-1}}(<\theta,\sigma>)&=& \int_0^1 \frac{t^{  n-1-\frac{2s+a}{p-1} }}{    (t^2+1- 2t <\theta,\sigma>)^{\frac{n+2s}{2}} }  dt +\int_1^\infty \frac{t^{  n-1-\frac{2s+a}{p-1} }}{    (t^2+1- 2t <\theta,\sigma>)^{\frac{n+2s}{2}} }  dt
\\&=& \int_0^1 \frac{t^{  n-1-\frac{2s+a}{p-1}} + t^{  2s-1+\frac{2s+a}{p-1}}}{    (t^2+1- 2t <\theta,\sigma>)^{\frac{n+2s}{2}} }  dt . 
 \end{eqnarray}
We now set $K_{\alpha}(<\theta,\sigma>)= \int_0^1 \frac{t^{  n-1+\alpha} + t^{  2s-1+\alpha}}{    (t^2+1- 2t <\theta,\sigma>)^{\frac{n+2s}{2}} }  dt$. The most important property of the $K_\alpha$ is that $K_\alpha$ is decreasing in $\alpha$. This can be seen by the following elementary calculations
   \begin{eqnarray}
\partial_\alpha K_\alpha&=& \int_0^1 \frac{-t^{  n-1-\alpha} \ln t + t^{  2s-1+\alpha} \ln t}{    (t^2+1- 2t <\theta,\sigma>)^{\frac{n+2s}{2}} }  dt
\\&=& \int_0^1 \frac{\ln t (-t^{  n-1-\alpha} + t^{  2s-1+\alpha})}{    (t^2+1- 2t <\theta,\sigma>)^{\frac{n+2s}{2}} }  dt<0 . 
 \end{eqnarray}
For the last part we have used the fact that for $p>\frac{n+2s+2a}{n-2s}$ we have $2s-1+\alpha<n-1-\alpha$.   From (\ref{Ans}) we get the following
 \begin{equation}\label{Ans2}
\int_{\mathbb S^{n-1}} \psi^2(\theta) A_{n,s,a} + \int_{\mathbb S^{n-1}} K_{\frac{2s+a}{p-1}} (<\theta,\sigma>) (\psi(\theta)-\psi(\sigma))^2 d\theta d \sigma= \int_{\mathbb S^{n-1}} \psi^{p+1}(\theta) d\theta . 
 \end{equation}
 We set a standard cut-off function $\eta_\epsilon\in C_c^1(\mathbb R_+)$ at the origin and at infinity that is $\eta_\epsilon=1$ for $\epsilon<r<\epsilon^{-1}$ and $\eta_\epsilon=0$ for either $r<\epsilon/2$ or $r>2/\epsilon$. We test the stability (\ref{stability}) on the function $\phi(x)=r^{-\frac{n-2s}{2}} \psi(\theta) \eta_\epsilon(r)$.    Note that
    \begin{eqnarray}
\int_{\mathbb R^n} \frac{\phi(x)-\phi(y)}{|x-y|^{n+2s}}  dy = \int \int_{\mathbb S^{n-1}} \frac{r^{-\frac{n-2s}{2} } \psi(\theta)\eta_\epsilon(r) -|y|^{-\frac{n-2s}{2} } \psi(\sigma) \eta_\epsilon(|y|)}{    (r^2+|y|^2- 2r |y| <\theta,\sigma>)^{\frac{n+2s}{2}} }  d\sigma d(|y|)  . 
\end{eqnarray}
Now set $|y|=rt$ then
 \begin{eqnarray*}
&&\int_{\mathbb R^n} \frac{\phi(x)-\phi(y)}{|x-y|^{n+2s}}  dy 
\\&=& r^{-\frac{n}{2} -s} \int_0^\infty \int_{\mathbb S^{n-1}} \frac{\psi(\theta)\eta_\epsilon(r) - t^{-\frac{n-2s}{2} } \psi(\sigma) \eta_\epsilon(rt)}{  (t^2+1- 2t <\theta,\sigma>)^{\frac{n+2s}{2}  }} t^{n-1} dt d\sigma
\\&=& r^{-\frac{n}{2} -s} \int \int_{\mathbb S^{n-1}} \frac{\psi(\theta)\eta_\epsilon(r) - t^{-\frac{n-2s}{2} } \psi(\sigma) \eta_\epsilon(r) +t^{-\frac{n-2s}{2} } (   \eta(r) \psi(\theta)-\eta_\epsilon(rt) \psi(\sigma)  )  }{  (t^2+1- 2t <\theta,\sigma>)^{\frac{n+2s}{2}  }} t^{n-1} dt d\sigma
\\&=& r^{-\frac{n}{2} -s} \eta_\epsilon(r) \psi(\theta) \int_0^\infty \int_{\mathbb S^{n-1}} \frac{1-t^{  \frac{n-2s}{2} }}{    (t^2+1- 2t <\theta,\sigma>)^{\frac{n+2s}{2}} }   t^{n-1} dt d\sigma
\\&&+ r^{-\frac{n}{2} -s} \eta_\epsilon(r)  \int_0^\infty \int_{\mathbb S^{n-1}} \frac{t^{ n-1 -\frac{n-2s}{2} } (\psi(\theta)-\psi(\sigma))}{    (t^2+1- 2t <\theta,\sigma>)^{\frac{n+2s}{2}} }  dt d\sigma
\\&&+  r^{-\frac{n}{2} -s} \int_0^\infty \int_{\mathbb S^{n-1}} \frac{t^{n-1  -\frac{n-2s}{2} } (\eta_\epsilon(r)-\eta_\epsilon(rt))\psi(\sigma) }{    (t^2+1- 2t <\theta,\sigma>)^{\frac{n+2s}{2}} } dt d\sigma . 
\end{eqnarray*}
Define $\Lambda_{n,s} :=\int_0^\infty  \int_{\mathbb S^{n-1}}\frac{1-t^{  \frac{n-2s}{2} }}{    (t^2+1- 2t <\theta,\sigma>)^{\frac{n+2s}{2}} }  t^{n-1} d\sigma dt$. Therefore,
 \begin{eqnarray}
\int_{\mathbb R^n} \frac{\phi(x)-\phi(y)}{|x-y|^{n+2s}}  dy &=&r^{-\frac{n}{2} -s} \eta_\epsilon(r) \psi(\theta) \Lambda_{n,s}
\\&&+ r^{-\frac{n}{2} -s} \eta_\epsilon(r)  \int_{\mathbb S^{n-1}} K_{\frac{n-2s}{2}}(<\theta,\sigma>) (\psi(\theta)-\psi(\sigma)) d\sigma
\\&&+r^{-\frac{n}{2} -s} \int_0^\infty \int_{\mathbb S^{n-1}} \frac{t^{  -\frac{n-2s}{2} } (\eta_\epsilon(r)-\eta_\epsilon(rt))\psi(\sigma) }{    (t^2+1- 2t <\theta,\sigma>)^{\frac{n+2s}{2}} } dt d\sigma . 
\end{eqnarray}
 Applying the above, we compute the left-hand side of the stability inequality (\ref{stability}),
\begin{eqnarray}\label{stable1}
&&\nonumber\int_{\mathbb R^n}\int_{\mathbb R^n} \frac{ (\phi(x)-\phi(y))^2}{|x-y|^{n+2s}} dx dy \\&=& 2\int_{\mathbb R^n}\int_{\mathbb R^n} \frac{ (\phi(x)-\phi(y))\phi(x)}{|x-y|^{n+2s}} dx dy
\nonumber\\&=&
2\int_0^\infty r^{-1} \eta_\epsilon^2(r) dr \int_{\mathbb S^{n-1}} \psi^2  \Lambda_{n,s} d\theta
\nonumber\\&&+ 2\int_0^\infty r^{-1} \eta_\epsilon^2(r) dr \int_{\mathbb S^{n-1}} K_{\frac{n-2s}{2}}(<\theta,\sigma>) (\psi(\theta)-\psi(\sigma))^2 d\sigma d\theta
\nonumber\\&&+2 \int_0^\infty  \left[ \int_0^\infty   r^{-1} \eta_\epsilon(r) (\eta_\epsilon(r)-\eta_\epsilon(rt)) dr  \right] \int_{\mathbb S^{n-1}} \int_{\mathbb S^{n-1}} \frac{   t^{ n- 1-\frac{n-2s}{2} } \psi(\sigma)\psi(\theta)  }{ (t^2+1- 2t <\theta,\sigma>)^{\frac{n+2s}{2}}}  d\sigma d\theta dt . 
 \end{eqnarray}
We now compute the second term in the stability inequality (\ref{stability}) for the test function $\phi(x)=r^{-\frac{n-2s}{2}} \psi(\theta) \eta_\epsilon(r)$ and $u=r^{-\frac{2s}{p-1}} \psi(\theta)$. So,  
\begin{eqnarray}\label{stable2}
\nonumber p \int_0^\infty r^a |u|^{p-1} \phi^2 &=& p \int_0^\infty r^a r^{-(2s+a)} r^{-(n-2s)} \psi^{p+1} \eta_\epsilon^2(r) dr \\&=&
p \int_0^\infty r^{-1} \eta_\epsilon^2(r)  dr \int_{\mathbb S^{n-1}} \psi^{p+1}  (\theta) d\theta . 
  \end{eqnarray}
Due to the definition of $\eta_\epsilon$, we have  $\int_0^\infty r^{-1} \eta_\epsilon^2(r)  dr=2 \ln (2/\epsilon) +O(1)$. Note that this quantity  appears in both terms of the stability inequality that we computed  in (\ref{stable1}) and (\ref{stable2}). We now claim that
 \begin{equation}
f_\epsilon(t):=\int_0^\infty   r^{-1} \eta_\epsilon(r) (\eta_\epsilon(r)-\eta_\epsilon(rt)) dr=O(\ln t).
 \end{equation}
Note that $\eta_\epsilon(rt)=1$ for $\frac\epsilon t<r<\frac{1}{t\epsilon}$ and $\eta_\epsilon(rt)=0$ for either $r<\frac{\epsilon}{2t}$ or $r>\frac{2}{t\epsilon}$. Now consider various ranges of value of $t\in (0,\infty)$ to compare the support of $\eta_\epsilon(r)$ and  $\eta_\epsilon(rt)$. From the definition of $\eta_\epsilon$, we have
 \begin{equation}
f_\epsilon(t)=\int_{\frac{\epsilon}{2}}^{ \frac{2}{\epsilon}    }   r^{-1} \eta_\epsilon(r) (\eta_\epsilon(r)-\eta_\epsilon(rt)) dr . 
 \end{equation}
In what follows we consider a few cases to explain  the claim. For example assume that  $\epsilon< \frac{\epsilon}{t}< \frac{1}{\epsilon}$ that holds when $\epsilon^2 <t<1$, then
 \begin{equation}
f_\epsilon(t)\approx \int_{\frac{\epsilon}{2}}^{ \frac{\epsilon}{t}    }   r^{-1} dr +  \int_{\frac{1}{\epsilon}}^{ \frac{2}{\epsilon t}    }   r^{-1} dr \approx \ln t .
 \end{equation}
Now let $\frac{1}{\epsilon}< \frac{\epsilon}{t}< \frac{2}{\epsilon}$ that holds when $\frac{\epsilon^2}{2}< t< \epsilon^2$.  The fact that $t \approx \epsilon^2$ implies that 
 \begin{equation}
f_\epsilon(t)\approx \int_{\frac{\epsilon}{2}}^{ \frac{\epsilon}{t}    }   r^{-1} dr +  \int_{\frac{\epsilon}{t}}^{ \frac{2}{\epsilon }    }   r^{-1} dr \approx \ln t + \ln \epsilon \approx \ln t.
 \end{equation}
Other cases can be treated similarly. From this one can see that
\begin{eqnarray*}
&&\int_0^\infty  \left[ \int_0^\infty   r^{-1} \eta(r) (\eta(r)-\eta(rt)) dr  \right] \int_{\mathbb S^{n-1}}\int_{\mathbb S^{n-1}} \frac{   t^{ n- 1-\frac{n-2s}{2} }   }{ (t^2+1- 2t <\theta,\sigma>)^{\frac{n+2s}{2}}} \psi(\sigma)\psi(\theta) d\sigma d\theta dt
\\&\approx&  \int_{\mathbb S^{n-1}} \int_{\mathbb S^{n-1}} \int_0^\infty \frac{   t^{ n- 1-\frac{n-2s}{2} }  \ln t }{ (t^2+1- 2t <\theta,\sigma>)^{\frac{n+2s}{2}}} \psi(\sigma)\psi(\theta)dt  d\sigma d\theta
\\&=& O(1) . 
  \end{eqnarray*}
 Collecting higher order terms  of the stability inequality we get
 \begin{equation}
  \Lambda_{n,s} \int_{\mathbb S^{n-1}} \psi^2 + \int_{\mathbb S^{n-1}} K_{\frac{n-2s}{2}}(<\theta,\sigma>) (\psi(\theta)-\psi(\sigma))^2 d\sigma \ge p  \int_{\mathbb S^{n-1}} \psi^{p+1} . 
 \end{equation}
 From this and (\ref{Ans2}) we obtain
  \begin{eqnarray}
  ( \Lambda_{n,s} -p A_{n,s,a}) \int_{\mathbb S^{n-1}} \psi^2 + \int_{\mathbb S^{n-1}} (K_{\frac{n-2s}{2}} - pK_{\frac{2s+a}{p-1}}  )(<\theta,\sigma>) (\psi(\theta)-\psi(\sigma))^2 d\sigma \ge 0  . 
  \end{eqnarray}
Note that $K_\alpha$ is decreasing in $\alpha$. This implies $K_{\frac{n-2s}{2}} < K_{\frac{2s+a}{p-1}}$ for $p>\frac{n+2s+2a}{n-2s}$. So, $K_{\frac{n-2s}{2}} - pK_{\frac{2s+a}{p-1}} <0$. On the other hand the assumption of the theorem implies that $\Lambda_{n,s} -p A_{n,s,a}<0$. Therefore, $\psi=0$.

 \end{proof}

 \section{Energy estimates}
 In this section, we provide some estimates for solutions of (\ref{main}). These estimates are needed in the next section when we perform a blow-down analysis argument.  The methods and ideas provided in this section are strongly motivated by \cite{ddww,ddw}.

 \begin{lemma} \label{bound} Let $u$ be a stable solution to \eqref{main}. Let also $\eta\in C^\infty_{c}(\R^n)$ and for $x\in\R^n$, define
\begin{equation} \label{def rho}
\rho(x) = \int_{\rn}\frac{(\eta(x)-\eta(y))^2}{\vert x-y\vert^{n+2s}}\;dy  . 
\end{equation}
Then,
 \begin{equation}
 \int_{\mathbb R^n} |x|^a |u|^{p+1} \eta^2 dx +  \int_{\mathbb R^n }\int_{\mathbb R^n } \frac{| u(x)\eta(x)-u(y)\eta(y)|^2}{|x-y|^{n+2s}} dx dy \le C \int_{\mathbb R^n} u^2 \rho dx  . 
 \end{equation}
\end{lemma}
\begin{proof}
 Proof is quite similar to Lemma 2.1 in \cite{ddw} and we omit it here.
  \end{proof}

 \begin{lemma}\label{rho}
 Let $m>n/2$ and $x\in\mathbb R^n$. Set
  \begin{equation}\label{eta}
 \rho(x)=\int_{\mathbb R^n} \frac{(\eta(x)-\eta(y))^2}{|x-y|^{n+2s}} dy \ \ \text{where} \ \ \eta(x)=(1+|x|^2)^{-m/2} . 
 \end{equation}
Then there is a constant $C=C(n,s,m)>0$ such that
\begin{equation}
 C^{-1} (1+|x|^2)^{-n/2-s}\le \rho(x)\le C (1+|x|^2)^{-n/2-s} . 
 \end{equation}
  \end{lemma}
 \begin{proof}
 Proof is quite similar to Lemma 2.2 in \cite{ddw} and we omit it here.
  \end{proof}

\begin{cor}\label{rhoR}
Suppose that $m>n/2$, $\eta$ given by (\ref{eta}) and $R>1$. Define
\begin{equation}\label{etaR}
 \rho_R(x)=\int_{\mathbb R^n} \frac{(\eta_R(x)-\eta_R(y))^2}{|x-y|^{n+2s}} dy \ \
 \text{where} \ \ \eta_R(x)=\eta(x/R) \phi(x) , 
  \end{equation}
 where $\phi\in C^{\infty}(\mathbb R^n)$ is a cut-off function such that $0 \le \phi\le 1$.   Then there exists a constant $C>0$ such that
 \begin{equation}
\rho_R(x)\le C \eta\left(\frac{x}{R}\right)^2 |x|^{-n-2s}+ R^{-2s} \rho\left(\frac{x}{R}\right).
 \end{equation}
\end{cor}

\begin{lemma}\label{finalu2}
Suppose that $u$ is a stable solution of (\ref{main}). Consider $\rho_R$ that is defined in Corollary \ref{rhoR} for $n/2<m<n/2+s(p+1)/2$. Then there exists a constant $C>0$ such that
 \begin{equation}
 \int_{\mathbb R^n} u^2 \rho_R \le C  R^{n- \frac{2s(p+1)+2a}{p-1}}  ,
\end{equation}
for any $R>1$. 
\end{lemma}
 \begin{proof}
Note that
\begin{eqnarray*}
\int_{\mathbb R^n} u^2 \rho_R dx \le \left( \int_{\mathbb R^n} |x|^{a} | |u|^{p+1} \eta_R^2 dx  \right)^{\frac{2}{p+1}}  \left( \int_{\mathbb R^n} |x|^{-\frac{2a}{p-1}}  \rho_R^{\frac{p+1}{p-1}} \eta_R^{-\frac{4}{p-1}} dx  \right)^{\frac{p-1}{p+1}} . 
\end{eqnarray*}
From Lemma \ref{bound} we get
\begin{eqnarray*}
\int_{\mathbb R^n} u^2 \rho_R dx \le  \int_{\mathbb R^n} |x|^{-\frac{2a}{p-1}}  \rho_R^{\frac{p+1}{p-1}} \eta_R^{-\frac{4}{p-1}} dx . 
\end{eqnarray*}
Now applying Corollary \ref{rhoR} for two different cases $|x|>R$ and $|x|<R$ one can get $\rho_R(x)\le C (|x|^{-n-2s}+R^{-2s})$ and 
 \begin{equation}
 \rho_R(x)\le C R^{-2s} \left( 1+\frac{|x|^2}{R^2} \right)^{-\frac{n}{2}-s}.
  \end{equation}
 This completes the proof.
  \end{proof}
 We now present  some more elliptic decay estimates on stable solutions. Since  proofs of these estimates are similar to the ones given   in \cite{ddw}, for the case of $0<s<1$, and  given in \cite{fw}, for the case of $1<s<2$, we omit them here.
\begin{lemma}\label{finalue2}
Suppose that $p\neq \frac{n+2s+2a}{n-2s}$. Let $u$ be a stable solution of (\ref{main}) and $u_e$ satisfies (\ref{maine}). Then there exists a constant $C>0$ independent from $R$ such that for  $0<s<1$ we have 
 \begin{equation}
 \int_{B_R} y^{1-2s} u_e^2 \le C R^{n+2- \frac{2s(p+1)+2a}{p-1}} , 
  \end{equation}
 and for $1<s<2$
we have 
 \begin{equation}
 \int_{B_R} y^{3-2s} u_e^2 \le C R^{n+4- \frac{2s(p+1)+2a}{p-1}}.  \end{equation}
\end{lemma}

\begin{lemma}\label{lowest} Let $u$ be a stable solution of (\ref{main}) and $u_e$ satisfies (\ref{maine}). Then there exists a positive constant $C$ independent from $R$ such that when   $0<s<1$; 
 \begin{equation}
\int_{B_R\cap\partial\mathbb R_+^{n+1}} |x|^a |u_e|^{p+1} dx + \int_{B_R\cap\mathbb R_+^{n+1}}  y^{1-2s} |\nabla u_e|^2 dxdy  \le C R^{n- \frac{2s(p+1)+2a}{p-1}}
  \end{equation} and when  $1<s<2$; 
\begin{equation}
\int_{B_R\cap\partial\mathbb R_+^{n+1}} |x|^a |u_e|^{p+1} dx + \int_{B_R\cap\mathbb R_+^{n+1}}  y^{3-2s} |\Delta_b u_e|^2 dxdy  \le C R^{n- \frac{2s(p+1)+2a}{p-1}} .
  \end{equation}
  \end{lemma}


\section{Blow-down analysis}
This section is devoted to the proof of Theorem \ref{mainthm}.  The methods and ideas are strongly motivated by the ones given in \cite{ddww,ddw}.
\\
\noindent{\it  Proof of Theorem \ref{mainthm}}: Let $u$ be a stable solution of (\ref{main}) and let $u_e$ be its extension solving \eqref{maine}. For the case $1<p\le p_{S}(n,a)$ the conclusion follows from the Pohozaev identity.  Note that for the subcritical case Lemma \ref{lowest} implies that $u\in \dot H ^s(\mathbb R^n)\cap L^{p+1}(\mathbb R^n)$. Multiplying (\ref{main}) with $u$ and doing integration, we obtain
\begin{equation}\label{poho1}\int_{\mathbb R^n} |x|^a u|^{p+1} = || u||^2_{\dot H^s(\mathbb R^n)}.
\end{equation}
In addition multiplying (\ref{main}) with $u^\lambda(x)=u(\lambda x)$ yields
\begin{equation}
 \int_{\mathbb R^n} |x|^a |u|^{p-1} u^\lambda = \int_{\mathbb R^n} (-\Delta )^{s/2} u (-\Delta)^{s/2} u^\lambda=\lambda^s \int_{\mathbb R^n} w w_\lambda,
 \end{equation}
 where $w=(-\Delta )^{s/2} u$.  Following ideas provided in \cite{ddww,rs} and  using the change of variable $z=\sqrt\lambda x$ one can get the following Pohozaev identity
\begin{equation}
 -\frac{n+a}{p+1} \int_{\mathbb R^n} |x|^a |u|^{p+1}= \frac{2s-n}{2} \int_{\mathbb R^n} w^2+ \frac{d}{d\lambda}\vert_{\lambda=1} \int_{\mathbb R^n} w^{\sqrt \lambda} w^{1/\sqrt\lambda} dz=\frac{2s-n}{2} ||u||^2_{\dot H^s(\mathbb R^n)}.
\end{equation}
This equality together and (\ref{poho1}) proves the theorem for the subcritical case.  Now suppose that $p> p_S(n,a)$. We consider two cases; 
\\
\\
{\bf Case 1: when $0<s<1$.} We perform the proof in a few steps.
\\
\noindent {\bf Step 1.} Boundedness of the limit: $\lim_{\lambda\to+\infty} E(u_e,\lambda)<+\infty.$ 
From the fact that $E$ is nondecreasing in $\lambda$, it suffices to show that $E(u_e,\lambda)$ is bounded. Write $E=I+J$, where $I$ is given by \eqref{I} and
\begin{equation}
J(u_e,\lambda) = \lambda^{\frac{2s(p+1)+2a}{p-1}-n-1}\frac{s+a}{p+1}\int_{\partial B_\lambda\cap\r}y^{1-2s}u_e^2\;d\sigma. 
\end{equation}
Note that Lemma \ref{lowest} implies that $I$ is bounded. To show that $E$ is bounded we state the following argument. The nondecreasing property of $E$ yields
\begin{equation}
E( u_e,\lambda ) \le\frac1\lambda \int_{\lambda}^{2\lambda} E(u,t) dt \le C +\lambda^{\frac{2s(p+1)+2a}{p-1}-n-1}\int_{B_{2\lambda}\cap\r}y^{1-2s} u_e^2.
\end{equation}
 From Lemma \ref{finalue2} we conclude that $E$ is bounded.
\\
\noindent {\bf Step 2.} There exists a sequence $\lambda_{i}\to+\infty$ such that $( u_e^{\lambda_{i}})$ converges weakly in $H^1_{loc}(\r;y^{1-2s}dydx)$ to a function $u_e^\infty$. This follows from the fact that $(u_e^{\lambda_{i}})$ is bounded in $H^1_{loc}(\r;y^{1-2s}dxdy)$ by Lemma \ref{lowest}.
\\
\noindent {\bf Step 3.\quad}$u_e^\infty$ is homogeneous.   To show this, we apply the scaling invariance property of $E$, its finiteness and the monotonicity formula. Let  $R_{2}>R_{1}>0$, then 
\begin{eqnarray*}
0
&=&\lim\limits_{n\to+\infty}E(u_e,\lambda_{i} R_{2})-E(u_e,\lambda_{i} R_{1})\\
&=&\lim\limits_{n\to+\infty}E(u_e^{\lambda_{i}},R_{2})-E( u_e^{\lambda_{i}},R_{1})\\
&\geq&\liminf\limits_{n\to+\infty}\int_{(B_{R_{2}}\setminus
B_{R_{1}})\cap\r}y^{1-2s}r^{2-n+\frac{4s+2a}{p-1}}\left(\frac{2s+a}{p-1}\frac{u_e^{\lambda_{i}}}{r} +\frac{\partial u_e^{\lambda_{i}}}{\partial r}\right)^2 dx dy\\
&\geq&\int_{(B_{R_{2}}\setminus
B_{R_{1}})\cap\r}y^{1-2s}r^{2-n+\frac{4s+2a}{p-1}}\left(\frac{2s+a}{p-1}\frac{ u_e^\infty}{r} +\frac{\partial
u_e^\infty}{\partial r}\right)^2  dx dy . 
\end{eqnarray*}
Note that in the last inequality we only used the weak convergence
of $(u_e^{\lambda_{i}})$ to $u_e^\infty$ in $H^1_{loc}(\r;y^{1-2s}dxdy)$. So,
\begin{equation}
\frac{2s+a}{p-1}\frac{ u_e^\infty}r
+\frac{\partial u_e^\infty}{\partial
r}=0\quad a.e.~~\text{in}~~\r.
\end{equation}
This implies that $u_e^\infty$ is homogeneous.
\\\noindent {\bf Step 4.} The function $ u_e^\infty$ must be identically zero.   This is in fact a direct consequence of Theorem \ref{homog}.
\\\noindent {\bf Step 5.} The functional sequence $( u_e^{\lambda_{i}})$ converges strongly to zero in $H^1(B_{R}\setminus B_{\eps};y^{1-2s}dxdy)$ and $(u^{\lambda_{i}})$ converges strongly to zero in $L^{p+1}(B_{R}\setminus B_{\eps})$ for all $R>\epsilon>0$.

From Step 2 and Step 3, we have $(u_e^{\lambda_{i}})$ is bounded in $H^1_{loc}(\r;y^{1-2s}dxdy)$ and converges weakly to $0$.  Therefore, $(u_e^{\lambda_{i}})$ converges strongly to zero in $L^2_{loc}(\r;y^{1-2s}dxdy)$. By the standard Rellich-Kondrachov theorem and a diagonal argument, passing to a subsequence, for any  $B_R = B_R(0) \subset \R^{n+1}$ and $A$ of the form  $A = \{ (x,t)\in\r: 0<t< r/2\}$, where $R,r>0$ we obtain
\begin{equation}
\lim_{i\to \infty  }\int_{\r \cap (B_R\setminus A)} y^{1-2s} | u_e^{\lambda_{i}}|^2 \, dx dy\to0.
\end{equation}
From  \cite[Theorem 1.2]{fks} we get 
\begin{equation}
\int_{\r\cap B_r(x)}y^{1-2s} |u_e^{\lambda_{i}}|^2 \, dx dy
\leq
C r^2 \int_{\r\cap B_r(x)}y^{1-2s} |\nabla u_e^{\lambda_{i}}|^2 \, dx dy, 
\end{equation}
for any $x\in \br$, $|x|\leq R$, with a uniform constant $C$. Applying similar arguments provided in \cite{ddw} one can show that  $( u_e^{\lambda_{i}})$ converges strongly to zero  in $H^1_{loc}(\r\setminus\{0\};y^{1-2s}dxdy)$ and the convergence also holds in  $L^{p+1}_{loc}(\R^n\setminus\{0\})$.
\\\noindent {\bf Step 6.} The function $u_e$ vanishes identically to zero. To prove this claim, consider 
\begin{eqnarray*}
I( u_e,\lambda) &=& I( u_e^\lambda,1)
\\&=&
\frac{1}{2}\int_{\r\cap B_{1}} y^{1-2s}\vert\nabla u_e^\lambda\vert^2  dxdy -\frac{\kappa_{s}}{p+1} \int_{\br\cap B_{1}} |x|^a \vert u_e^\lambda\vert^{p+1}dx
\\&=&\frac{1}{2} \int_{\r\cap B_{\epsilon}} y^{1-2s} \vert\nabla u_e^\lambda\vert^2 dx dy - \frac{\kappa_{s}}{p+1} \int_{\br\cap B_{\epsilon}}  |x|^a\vert u_e^\lambda\vert^{p+1}dx
\\&&+\frac{1}{2} \int_{\r\cap B_{1}\setminus B_{\epsilon}} y^{1-2s} \vert\nabla u_e^\lambda\vert^2 dx dy - \frac{\kappa_{s}}{p+1} \int_{\br\cap B_{1}\setminus B_{\epsilon}} |x|^a \vert u_e^\lambda\vert^{p+1}dx
\\&=&\eps^{n-\frac{2s(p+1)+2a}{p-1}} I(u_e,0,\lambda\eps) +\frac{1}{2} \int_{\r\cap B_{1}\setminus B_{\epsilon}} y^{1-2s} \vert\nabla u_e^\lambda\vert^2 dx dy -\frac{\kappa_{s}}{p+1} \int_{\br\cap B_{1}\setminus B_{\epsilon}} |x|^a  \vert  u_e^\lambda\vert^{p+1}dx\\
&\le& C\eps^{n-\frac{2s(p+1)+2a}{p-1}} + \frac{1}{2} \int_{\r\cap B_{1}\setminus B_{\epsilon}} y^{1-2s} \vert\nabla u_e^\lambda\vert^2 dx dy - \frac{\kappa_{s}}{p+1} \int_{\br\cap B_{1}\setminus B_{\epsilon}}   |x|^a \vert  u_e^\lambda\vert^{p+1}dx.
\end{eqnarray*}
Letting $\lambda\to+\infty$ and then $\eps\to0$, we deduce that $\lim_{\lambda\to+\infty} I( u_e,\lambda) =0.$ Using the monotonicity property of $E$, we get 
\begin{equation}
E(u_e,\lambda) \le \frac1\lambda\int_{\lambda}^{2\lambda}E(t)\;dt\le \sup_{[\lambda,2\lambda]}I + C\lambda^{-n-1+\frac{2s(p+1)+2s}{p-1}}\int_{B_{2\lambda}\setminus B_{\lambda}}u_e^2. 
\end{equation}
Therefore, $\lim_{\lambda\to+\infty}E(u_e,\lambda) =0.$ Note that $u$ is smooth and also  $E(u_e,0)=0$. Since $E$ is monotone, $E$ is identically zero.  Therefore, $u_e$ must be homogeneous that is a contradiction unless $u_e\equiv0$.
\\
\\
{\bf Case 2: when $1<s<2$.} Proof of this case is very similar to Case 1.  We perform the proof in a few steps.
\\
\noindent {\bf Step 1.} Finiteness of the limit:  $\lim_{\lambda\to\infty} E(u_e,\lambda)<\infty$.    Theorem \ref{mono} implies that  $E$ is nondecreasing. So, we only need to show that $E(u_e,\lambda)$ is bounded.  Note that
\begin{equation}
E(u_e,\lambda) \le \frac{1}{\lambda} \int_\lambda^{2\lambda}  E(u_e,t) dt  \le \frac{1}{\lambda^2} \int_\lambda^{2\lambda} \int_t^{t+\lambda} E(u_e,\gamma) d\gamma dt .
\end{equation}
From Lemma \ref{lowest} we conclude that
\begin{equation}
 \frac{1}{\lambda^2} \int_\lambda^{2\lambda} \int_t^{t+\lambda}   \gamma^{2s\frac{p+1}{p-1}-n} \left(   \int_{  \mathbb{R}^{n+1}_{+}\cap B_\gamma} \frac{1}{2} y^{3-2s}|\Delta_b u_e|^2 dy dx-  \frac{C_{n,s}}{p+1} \int_{  \partial\mathbb{R}^{n+1}_{+}\cap B_\gamma} |x|^a u_e^{p+1}  dx \right)  d\gamma dt \le C,
 \end{equation}
where $C>0$ is independent from $\lambda$. For the other term in the energy we have
\begin{eqnarray*}\label{}
&&\frac{1}{\lambda^2} \int_\lambda^{2\lambda} \int_t^{t+\lambda}  \left(\gamma^{-3+2s+\frac{4s+2a}{p-1}-n}  \int_{  \mathbb{R}^{n+1}_{+}\cap \partial B_\gamma} y^{3-2s} u_e^2 dydx\right)d\gamma dt
\\ &\le&\frac{1}{\lambda^2} \int_\lambda^{2\lambda} t^{-3+2s+\frac{4s+2a}{p-1}-n} \int_{B_{t+\lambda}\setminus B_t}   y^{3-2s} u_e^2 dydx  dt
\\&\le& \frac{1}{\lambda^2} \int_\lambda^{2\lambda} t^{-3+2s+\frac{4s+2a}{p-1}-n}\left( \int_{B_{3\lambda}}   y^{3-2s} u_e^2 dydx \right) dt
\\&\le& \lambda^{n+4-\frac{2s(p+1)+2a}{p-1}}\frac{1}{\lambda^2} \int_\lambda^{2\lambda} t^{-3+2s+\frac{4s+2a}{p-1}-n} dt
\\&\le& C, 
\end{eqnarray*}
 and $C>0$ is independent from $\lambda$. Note that to deduce above estimates we applied  Lemma \ref{finalue2}.  For the other term in the energy, we have
 \begin{eqnarray*}\label{}
&&\frac{1}{\lambda^2} \int_\lambda^{2\lambda} \int_t^{t+\lambda} \frac{\gamma^3}{2} \frac{d}{d \gamma} \left[ \gamma^{2s-3-n+\frac{4s+2a}{p-1}} \int_{\partial B_\gamma} y^{3-2s} \left( \frac{2s+a}{p-1} \gamma^{-1} u_e+\frac{\partial u_e}{\partial r} \right)^2  \right] d\gamma dt
\\&=& \frac{1}{2\lambda^2} \int_\lambda^{2\lambda} [(t+\lambda)^{2s-n+\frac{4s+2a}{p-1}} \int_{\partial B_{t+\lambda}} y^{3-2s} \left( \frac{2s+a}{p-1} (t+\lambda)^{-1} u_e+\frac{\partial u_e}{\partial r} \right)^2
\\&&
- t^{2s-n+\frac{4s+2a}{p-1}}  \int_{\partial B_{\lambda}} y^{3-2s} \left( \frac{2s+a}{p-1} \gamma^{-1} u_e+\frac{\partial u_e}{\partial r} \right)^2 ] dt
\\&&-\frac{3}{2\lambda^2} \int_\lambda^{2\lambda} \int_t^{t+\lambda}\left[ \gamma^{2s-1-n+\frac{4s+2a}{p-1}}  \int_{\partial B_{\gamma}}  y^{3-2s} \left( \frac{2s+a}{p-1} \gamma^{-1} u_e+\frac{\partial u_e}{\partial r} \right)^2\right] d\gamma dt
\\&\le& \lambda^{-2+2s -n +\frac{4s+2a}{p-1}} \int_{B_{3\lambda}\setminus B_\lambda} y^{3-2s} \left( \frac{2s+a}{p-1} \lambda^{-1} u_e+\frac{\partial u_e}{\partial r} \right)^2 \le C ,
 \end{eqnarray*}
and again  $C>0$ is independent from $\lambda$. The rest of the terms can be treated similarly.
 \\
 \noindent {\bf Step 2.}  There exists a sequence $\lambda_i\to\infty$ such that $(u_e^{\lambda_i})$ converges weakly in $H^1_{loc}(\mathbb R^n, y^{3-2s} dxdy)$ to a function $u_e^\infty$.   Note that this is in fact a direct consequence of Lemma \ref{lowest}.
 \\
 \noindent {\bf Step 3.} The limit function $u_e^\infty$ is homogeneous and therefore it must be identically zero.  To prove this claim, we apply the scaling invariance property of $E$, its finiteness and the monotonicity formula, just like in the case of $0<s<1$. Suppose that  $R_2>R_1>0$, then 
  \begin{eqnarray*}\label{}
  0&=& \lim_{i\to\infty} \left(E(u_e,R_2\lambda_i)- E(u_e,R_1\lambda_i) \right)
\\&=&    \lim_{i\to\infty} \left(E(u_e^{\lambda_i},R_2)- E(u_e^{\lambda_i},R_1) \right)
\\&\ge&  \liminf_{i\to\infty} \int_{(B_{R_2} \setminus B_{R_1})\cap \mathbb R^{n+1}_+} y^{3-2s}  r^{\frac{4s+2a}{p-1}+2s-2-n}   \left(  \frac{2s+a}{p-1} r^{-1} u_e^{\lambda_i}+ \frac{\partial u_e^{\lambda_i}}{\partial r}\right)^2 dy dx
\\&\ge&  \int_{(B_{R_2} \setminus B_{R_1})\cap \mathbb R^{n+1}_+} y^{3-2s}  r^{\frac{4s+2a}{p-1}+2s-2-n}   \left(  \frac{2s+a}{p-1} r^{-1} u_e^{\infty}+ \frac{\partial u_e^{\infty}}{\partial r}\right)^2 dy dx
   \end{eqnarray*}
In the last inequality we have used the weak convergence of $(u_e^{\lambda_i})$ to $u_e^\infty$ in $H^1_{loc}(\mathbb R^n,y^{3-2s} dydx)$. This implies
\begin{equation}
  \frac{2s+a}{p-1} r^{-1} u_e^{\infty}+ \frac{\partial u_e^{\infty}}{\partial r}=0 \ \ \text{a.e. \ \ in} \ \ \mathbb R_+^{n+1}.
  \end{equation}
  Therefore, $u_e^\infty$ is homogeneous.  Apply Theorem \ref{homog} we get  $u_e^\infty=0$.
   \\
 \noindent {\bf Step 5.} The sequence $(u_e^{\lambda_i})$ converges strongly to zero in $H^1(B_R\setminus B_\epsilon, y^{3-2s} dydx)$ and $(u_e^{\lambda_i})$ converges strongly to zero in $L^{p+1}(B_R\setminus B_\epsilon)$ for all $R>\epsilon>0$.
\\\noindent {\bf Step 6.} Finally we claim that $u_e$ must be identically zero. We now apply the scaling invariance property and also elliptic estimates to show that 
\begin{eqnarray*}
I( u_e,\lambda) &=& I( u_e^\lambda,1)
\\&=&
\frac{1}{2}\int_{\r\cap B_{1}} y^{3-2s}\vert\Delta_b u_e^\lambda\vert^2  dxdy -\frac{\kappa_{s}}{p+1} \int_{\br\cap B_{1}} |x|^a \vert u_e^\lambda\vert^{p+1}dx
\\&=&\frac{1}{2} \int_{\r\cap B_{\epsilon}} y^{3-2s} \vert \Delta_b u_e^\lambda\vert^2 dx dy - \frac{\kappa_{s}}{p+1} \int_{\br\cap B_{\epsilon}}  |x|^a\vert u_e^\lambda\vert^{p+1}dx
\\&&+\frac{1}{2} \int_{\r\cap B_{1}\setminus B_{\epsilon}} y^{3-2s} \vert\Delta_b u_e^\lambda\vert^2 dx dy - \frac{\kappa_{s}}{p+1} \int_{\br\cap B_{1}\setminus B_{\epsilon}} |x|^a \vert u_e^\lambda\vert^{p+1}dx
\\&=&\eps^{n-\frac{2s(p+1)+2a}{p-1}} I(u_e,\lambda\eps) +\frac{1}{2} \int_{\r\cap B_{1}\setminus B_{\epsilon}} y^{3-2s} \vert\Delta_b u_e^\lambda\vert^2 dx dy -\frac{\kappa_{s}}{p+1} \int_{\br\cap B_{1}\setminus B_{\epsilon}} |x|^a  \vert  u_e^\lambda\vert^{p+1}dx\\
&\le& C\eps^{n-\frac{2s(p+1)+2a}{p-1}} + \frac{1}{2} \int_{\r\cap B_{1}\setminus B_{\epsilon}} y^{3-2s} \vert\Delta u_e^\lambda\vert^2 dx dy - \frac{\kappa_{s}}{p+1} \int_{\br\cap B_{1}\setminus B_{\epsilon}}   |x|^a \vert  u_e^\lambda\vert^{p+1}dx. 
\end{eqnarray*}
Letting $\lambda\to+\infty$ and then $\eps\to0$, we conclude 
$
\lim_{\lambda\to+\infty} I( u_e,\lambda) =0.
$
Using the monotonicity property of $E$, we obtain 
\begin{equation}
E(u_e,\lambda) \le \frac1\lambda\int_{\lambda}^{2\lambda}E(t)\;dt\le \sup_{[\lambda,2\lambda]}I + C\lambda^{-n-1+\frac{2s(p+1)+2a}{p-1}}\int_{B_{2\lambda}\setminus B_{\lambda}}u_e^2\\
\end{equation}
and so
$
\lim_{\lambda\to+\infty}E(u_e,\lambda) =0.$ Note that  $u$ is smooth and $E(u_e,0)=0$. Since $E$ is monotone, $E$ must be identically zero. This implies that $ u_e$ must be homogeneous. Therefore, $u_e$ must vanishes identically to zero. 
  
   \hfill $ \Box$ 
\\
\\
\\
{\it\bf  Acknowledgement.} The first author is grateful to Pacific Institute for Mathematical Sciences (PIMS) at UBC for the hospitality during his visits.

\end{document}